\documentclass[11pt,a4paper]{amsart}
\usepackage{amsmath,amsfonts,xargs,microtype,amssymb,bbm}
\usepackage[utf8]{inputenc}
\usepackage{graphicx} 
\usepackage{amsthm}
\usepackage{tikz}
\usetikzlibrary{calc}
\usepackage{tkz-tab}
\usepackage{theoremref}
\usepackage{thmtools, thm-restate}
\declaretheorem{theorem}
\usepackage{thm-restate}
\usepackage{hyperref}
\usepackage{comment}
\usepackage{color}
\usepackage{mathtools}
\usepackage{subcaption}
\usepackage{hyphenat}
\mathtoolsset{showonlyrefs}

\hypersetup{colorlinks=true,urlcolor=black, pdftitle="LAB-R2"}

\allowdisplaybreaks

\begin{document}

\newcommand{\R}{ {\mathbb R}}
\newcommand{\N}{ {\mathbb N}}
\newcommand{\s}{ {\mathbb S}}
\newcommand{\conv}{{\operatorname{conv}}}
\newcommand\numberthis{\addtocounter{equation}{1}\tag{\theequation}}

\theoremstyle{plain}

\newtheorem{lemma}[theorem]{Lemma}
\newtheorem{corollary}[theorem]{Corollary}
\newtheorem{prop}[theorem]{Proposition}

\newtheorem{conj}{Conjecture}

\theoremstyle{remark}
\newtheorem{remark}{Remark}

\newcommand{\correspondingA}{\@myfnsymbol{1}}
\newcommand{\correspondingB}{\@myfnsymbol{2}}
\newcommand{\correspondingC}{\@myfnsymbol{3}}
\newcommand{\correspondingD}{\@myfnsymbol{4}}
\newcommand{\correspondingE}{\@myfnsymbol{5}}
\makeatother



\title[Equality cases for the $L_p$-Rogers--Shephard inequality]{Equality cases for the $L_p$-Rogers--Shephard inequality in the plane and for locally anti-blocking bodies in $\mathbb{R}^n$}

\author[M. Fradelizi, A. Manui, M. Meyer, C. S. Ndiaye]{Matthieu Fradelizi, Auttawich Manui, Mark Meyer, and Cheikh Saliou Ndiaye}


\date{}

\begin{abstract} 
The classical Rogers--Shephard inequalities were extended to the Firey $L_p$-summation by Bianchini and Colesanti in the plane and by Zvavitch and the second and fourth authors for locally anti-blocking convex bodies in $\R^n$, leaving open the equality cases. 
    We characterize the equality cases of these inequalities: in both cases, for $p>1$, equality holds if and only if the convex body is a simplex with one vertex at the origin.
\end{abstract}

\maketitle

\section{Introduction}

A \textit{convex body} is a subset of $\mathbb{R}^n$ that is compact and convex with nonempty interior, and the \textit{Minkowski sum} of convex bodies $K,L\subset\mathbb{R}^n$ is the set
\begin{equation*}
    K+L:=\{k+l:k\in K,l\in L\}.
\end{equation*}
In a 1957 paper \cite{RS-57}, Rogers and Shephard proved the following sharp volume bound: if $K\subset\mathbb{R}^n$ is a convex body such that $0\in K$, and if $-K:=\{-k:k\in K\}$ is the reflection of $K$ about the origin, then the volume of the difference body $K-K:=K+(-K)$ satisfies the bound
\begin{align}\label{eq:rogers_shephard_minkowski_sum}
    |K-K|\leq \binom{2n}{n}|K|,
\end{align}
with equality if and only if $K$ is a simplex. In another paper \cite{RS-58}, Rogers and Shephard proved: if $K\subset\mathbb{R}^n$ is a convex body such that $0\in K$, then the convex hull of the union $K\cup (-K)$ satisfies the bound
\begin{equation}\label{eq:rogers_shephard_convex_hull}
    |\textup{conv}(K\cup(-K))|\leq 2^n|K|,
\end{equation}
which is also sharp, and equality holds if and only if $K$ is a simplex with one vertex at the origin.

To generalize the above inequalities, we begin by recalling that the support function $h_K$ of a convex set $K\subset\mathbb{R}^n$ is the function that satisfies
\begin{equation*} 
    h_K(x):=\sup_{k\in K}\langle x,k\rangle,\qquad x\in \mathbb{R}^n.
\end{equation*}
If $K,L\subset \mathbb{R}^n$ are convex sets containing the origin, then for any fixed $p\in [1,\infty)$ there exists a convex set $K\oplus_p L$, containing the origin whose support function satisfies
\begin{align}\label{def:Firey-sums}
    h_{K \oplus_p L}(x)=\left(h_K^p(x)+h_L^p(x)\right)^{\frac{1}{p}}.
\end{align}
We can take the limit as $p\rightarrow\infty$ to arrive at the natural definition $h_{K\oplus_{\infty} L}=\max\{h_K,h_L\}$. The set $K\oplus_p L$ is the \textit{Firey sum}, or \textit{$L_p$-sum} of the convex bodies $K$ and $L$. The concept of Firey sum originates from the work of Firey \cite{F-62}. Of special interest is the fact that
\begin{align*}
    K\oplus_1L=K+L,\qquad \textup{and}\qquad K\oplus_{\infty} L=\textup{conv}(K\cup L),
\end{align*}
so we can rephrase the bounds \eqref{eq:rogers_shephard_minkowski_sum} and \eqref{eq:rogers_shephard_convex_hull} in terms of the Firey sum for the cases $p=1$ and $p=\infty$.

To generalize the work of Rogers and Shephard in the setting of $L_p$-sums, it is natural to conjecture that if $K\subset\mathbb{R}^n$ is a convex body with $0\in K$, then 
\begin{equation}\label{eq:rogers_shephard_conjecture}
    |K\oplus_p(-K)|\leq \kappa_{n,p}|K|,\qquad \textup{where}\qquad \kappa_{n,p}:=\sum_{i=0}^{n}\binom{n/q}{i/q}^{-1}\binom{n}{i}^2,
\end{equation}
where $q$ satisfies $\frac{1}{p}+\frac{1}{q}=1$, and equality holds for $p>1$ if and only if $K$ is a simplex with one vertex at the origin. From the fact that
\begin{equation*}
    \kappa_{n,1}=\binom{2n}{n}\qquad \textup{and}\qquad \kappa_{n,\infty}=2^n,
\end{equation*}
we see that conjecture \eqref{eq:rogers_shephard_conjecture} reduces to the original inequalities of Rogers and Shephard when $p=1$ and $p=\infty$.

Bianchini and Colesanti \cite[Theorem 1.1]{BC-08} were able to prove inequality \eqref{eq:rogers_shephard_conjecture} in the planar case $n=2$ by making a key observation about parallel chord movements of convex bodies. Zvavitch, together with the second and fourth named authors of this paper, managed to show that inequality \eqref{eq:rogers_shephard_conjecture} is true in dimension $n\geq 2$ if one restricts to locally anti-blocking convex bodies \cite[Lemma 23]{MNZ-25}, which are the convex bodies for which all orthogonal projections onto and intersections with the coordinate hyperplanes are identical. Moreover, they established that equality holds among anti-blocking convex bodies if and only if the body is a simplex. Recall that an anti-blocking convex body is a locally anti-blocking convex body included in the positive orthant.

In the forthcoming paper \cite{FMMN-26-2}, the authors of the present paper show that conjecture~\eqref{eq:rogers_shephard_conjecture} is true for $L_2$-zonoids, and even more, that in this setting it is possible to replace the constant $\kappa_{n,p}$ with the improvement $2^n$.

The purpose of this paper is to fill in two gaps in the literature by characterizing the conditions of equality for the planar inequality of Bianchini and Colesanti, and the locally anti-blocking inequality of Zvavitch, with the second and fourth named authors of this paper. In the case of locally anti-blocking convex bodies, we will show:

\begin{theorem}\label{thm:anti-blocking_eq}
    Let $p\in (1,+\infty]$, and let $K\subset\mathbb{R}^n$ be a locally anti-blocking convex body. Then
    \begin{equation*}
        |K\oplus_p (-K)|=\kappa_{n,p}|K|
    \end{equation*}
    holds if and only if $K$ is a simplex of the form $K=\textup{conv}(0,\alpha_1e_1,\dots,\alpha_ne_n)$ for non zero real numbers $\alpha_i$.
\end{theorem}

In the planar case, we will make use of a technique of Mathieu Meyer \cite{M-91} to prove:

\begin{theorem}\label{thm:Planar-case-general}
    Let $p\in (1,+\infty]$, and let $K\subset\mathbb{R}^2$ be a convex body such that $0\in K$. Then
    \begin{equation*}
        |K\oplus_p(-K)|=\kappa_{2,p}|K|
    \end{equation*}
    holds if and only if $K$ is a triangle with one vertex at the origin.
\end{theorem}

We dedicate the remainder of the paper to the proof of the above two theorems. We will give a review of relevant preliminary facts in Section \ref{sec:preliminarily}. The following sections, Sections \ref{subsec:lab} and \ref{subsec:r2}, cover the proofs of Theorems \ref{thm:anti-blocking_eq} and \ref{thm:Planar-case-general}.

\subsection{Related literature}

The reader may consult the surveys of Lutwak \cite{L-93,L-96} for a general overview of the $L_p$-Brunn--Minkowski theory.
We refer \cite{BC-08, CCG-99,CG-06,CG-02,CG-06-02} for background on shadow systems

Aside from the results of this paper, there is some interest in the study of the Rogers--Shephard inequality in a general setting. For example, the paper \cite{AHRYZ-21} studies the Rogers--Shephard inequality for general measures, and the papers \cite{A-19,AGJV-16,C-06,AAGJV-19,MMZZ-26} focus on functional analogues.

The papers \cite{K-25,CFS-17,F-71} cover recent advancements in the study of locally anti-blocking convex bodies, which is a topic of independent interest.

\subsection{Funding} 
The second-named author is supported by the Chateaubriand Fellowship of the Office for Science \& Technology of the Embassy of France in the United States, the U.S. National Science Foundation Grant DMS-2247771, and the United States-Israel Binational Science Foundation (BSF) Grant 2018115.

The third-named author is supported by the
National Science Foundation through the MSPRF program (award number: 2502794).

\section{Preliminaries} 
    \label{sec:preliminarily}
    We refer to \cite{S-93,L-93,L-96} for general background material. We write $\langle x , y \rangle$ for the inner product of vectors $x,y$. For a set $K \subset \mathbb{R}^n$, we denote by $|K|_m$ the $m$-dimensional Lebesgue measure of $K$ restricted to its affine hull, which we denote by $\textup{aff} (K)$. When the dimension of $K$ is equal to that of the ambient space, we simply write $|K|$. 
    
    We recall the extension of the $L_p$-sum \eqref{def:Firey-sums} to nonconvex sets, introduced by Lutwak, Yang, and Zhang \cite{LYZ-12}. For $p>1$ and sets $K,L\subset\mathbb{R}^n$, define
    \begin{equation}
        \label{def:Firey-sums-non-convex}
        K \oplus_p L:=\left\{(1-t)^\frac{1}{q} x+t^\frac{1}{q} y \mid x \in K, y \in L, t \in[0,1]\right\},
    \end{equation}
    where $q$ satisfies $1/q+1/p=1$. Definition \eqref{def:Firey-sums-non-convex} coincides with \eqref{def:Firey-sums} when $K$ and $L$ are convex compact sets containing the origin, as shown in \cite[Lemma 2]{LYZ-12}. Note also that the $L_p$-sum commutes with linear transformations. That is, for any convex sets $K,L\subset\mathbb{R}^n$ containing origin and any linear transformation $T$, we have
    \begin{equation}
        T (K\oplus_p L) = (TK) \oplus_p (TL).
        \label{eq:linear-invariant-lp}
    \end{equation}
We recall that a convex set $K\subset \R^n$ is locally anti-blocking if for any coordinate subspace $E$, $P_E K = K \cap E$. A convex set $K$ is anti-blocking if it is a locally anti-blocking subset of $\R^n_+$.
One of the main tools in \cite{MNZ-25} is the  Sch\"utellung (or shaking) symmetrization of a convex body $K$ in the direction $i$, which is defined as 
\begin{equation}
    S_i(K):=\{x+te_i:x\in P_{e_i^{\bot}}K, 0\leq t\leq |K\cap \mathbb{R}e_i|_1\},
\end{equation}
and is a special shadow system.

A shadow system  along a direction $u$ is a family of convex bodies $K(u)$ in $\R^n$ obtained by projecting a convex body $\Tilde{K}$ in $\R^{n+1}$ onto $e_{n+1}^\perp$ along a direction $e_{n+1} + u$. Here, $ u $ varies in $e_{n+1}^\perp$.
A linear parameter system along a direction $u$ is a family of convex bodies $\{K_t\}_{t \in I}$ in $\R^n$ that can be written in the form 
\[
    K_t = \conv \{ z + \alpha(z) tu : z\in A \}, 
\]
where $A$ is a bounded set, the speed function $\alpha$ is  a real-valued function on $A$, and the parameter $t$ ranges over on an interval $I\subset \R$. We note the linear parameter is a special case of the shadow system in which $u$ lies on a line, see \cite[Proposition 2.1]{BC-08}. It was proved in \cite{RS-58-2} that the volume of a shadow system is a convex function in $t $ on $ I$. The proof is, in fact, based on that for each line $\ell$ parallel to $u$, the chord length $t\mapsto |\ell\cap K_t|$ is a convex function in $t$, see \cite[Section 3]{CG-06-02}.

The following theorem was proved by Bianchini and Colestanti \cite[Theorem 2.3]{BC-08} and is one of their main ingredients in the proof of inequality \eqref{eq:rogers_shephard_conjecture} for the planar case.
\begin{theorem} \label{thm:BC-shadow-lp}
    Let $\{K_t\}_{t \in I}$ and $\{L_t\}_{t \in I}$ be linear parameter systems along the direction $u$ and let $p \geq 1$. Then, $\{ K_t \oplus_p L_t\}_{t \in I}$ is a linear parameter system along the direction $u$.
\end{theorem}
We are particularly interested in the case of the linear parameter systems for which the speed function $\alpha$ is constant on each chord of $K$ parallel to $u$, so-called the parallel chord movement. More precisely, a parallel chord movement along the direction $u$ is a family of convex bodies $\{K_t\}_{t\in I}$ that can be defined in the form 
\[
    K_t = \conv\{ z+\beta(P_{u^\perp} z) tu : z \in K\}, 
\]
where $\beta $ is a bounded real-valued function defined on $P_{u^\perp} K$, and $I$ is an interval of $\R$. We note that along the parallel chord movement, the volume $K_t$ does not change.
Well-known examples include the Sch\"utellung process and the Steiner symmetrization.

Finally, recall that for a convex body $K \subset \R^n$ and a direction $u \in \s^{n-1}$, one can present $K$ in the following form
\begin{equation}
    K = \{ x+ su: x \in P_{u^\perp} K, -g(x) \leq s \leq f(x)\},
\end{equation}
for some concave functions $f,g: P_{u^\perp} K\to\R$.

\section{Proof of Theorem \ref{thm:anti-blocking_eq}} \label{subsec:lab}



\begin{lemma}\label{strict_containment_implies_strict_bound}
Let $A,B\subset\mathbb{R}^n_+$ be anti-blocking convex compact sets such that $A\oplus_p -B$ is full-dimensional. 
Then 
\begin{equation}\label{volume_orthant_lemma}
        |A\oplus_p B|
    \le|A\oplus_p -B|.
    \end{equation}
Moreover, if there is equality, then for all $i\in [n]$ 
    \begin{equation}\label{strict_symmetrization_containment}
        S_i(A)\oplus_p S_i(-B)= S_i(A\oplus_p -B).
    \end{equation}    
\end{lemma}
\begin{proof}
To study the equality case, we first need to recall the proof of the inequality from \cite{MNZ-25}. In \cite[Lemma 18]{MNZ-25}, it is established that for any convex sets $K,L$
\begin{equation}\label{symmetrization_containment}    S_i(K)\oplus_p S_i(L)\subset S_i(K\oplus_p L).\end{equation}
To prove the bound \eqref{volume_orthant_lemma}, we recall the chain of inequalities established in \cite[Lemma 19]{MNZ-25}:
    \begin{equation}\label{n_fold_symmetrization}
        |A\oplus_p -B|=|S_n(A\oplus_p -B)|\geq |S_n(A)\oplus_p S_n(-B)| = |A\oplus_p S_n(-B)|,
    \end{equation}
    where 
    the first equality in \eqref{n_fold_symmetrization} is due to the fact that the operator $S_n$ preserves volume and the last equality follows from $S_n (A) =A$. We repeat the process by applying the other operators $S_j$ for $ j <n$ to obtain 
    \begin{equation}
        | A \oplus_p -B| \geq |A\oplus_p S_n(-B)|\ge\cdots \geq |A\oplus_p S_1\cdots  S_{n-1}S_n(-B)|  = |A \oplus_p B|,
        \label{n_fold_symmetrization-2}
    \end{equation}
    where the last equality is true because $S_1\ldots S_{n-1}S_n(-B) =B$. Now, if $|A\oplus_p B|=|A\oplus_p -B|$, then there is equality in all the inequalities of the line above. In particular, $|S_n(A \oplus_p -B)|=|S_n(A)\oplus_p S_n(-B)|$. But, since one set is included in the other and both are convex bodies, this implies that $S_n(A \oplus_p -B)=S_n(A)\oplus_p S_n(-B)$. Notice that the shakings in different coordinate directions commute for $A$ and $-B$, because they are just symmetries with respect to coordinate hyperplanes. So we can use the same argument to conclude that $ S_i(A)\oplus_p S_i(-B)= S_i(A\oplus_p -B)$ for all $i$.
\end{proof}

The next lemma characterizes the conditions of equality for the bound \eqref{volume_orthant_lemma}.

\begin{lemma}\label{reflection_lemma}
    Let $A,B\subset\mathbb{R}^n_+$ be anti-blocking convex sets. Then, for $1<p\leq \infty$ the equality
    \begin{equation}\label{reflection_eq}
        |A\oplus_pB|=|A\oplus_p-B|
    \end{equation}
    holds if and only if either $\textup{aff}(A)\cap \textup{aff}(B)=\{0\}$, or there exists $i\in [n]$ such that both of the sets $A$ and $B$ are contained in $e_i^{\bot}$.
\end{lemma}

\begin{proof}
    Suppose that there exists $i\in [n]$ such that both $A$ and $B$ are contained in $e_i^{\bot}$. Then $A\oplus_pB$ and $A\oplus_p-B$ are contained in $e_i^{\bot}$, and \eqref{reflection_eq} is the trivial equality $0=0$. Next, suppose that $\textup{aff}(A)\cap \textup{aff}(B)=\{0\}$. For $x_A\in \textup{aff}(A)$ and $x_B\in \textup{aff}(B)$, define the reflection $R:\mathbb{R}^n\rightarrow\mathbb{R}^n$ by $R(x_A+x_B)=x_A-x_B$. Then $R$ is linear, preserves volume, and satisfies the relations $R(A)=A$, $R(B)=-B$. It follows that
    \begin{equation*}
        \begin{split}
            |A\oplus_pB|&=\left|R\left(A \oplus_p B\right)\right| = |R(A) \oplus_p R(B)| =|A\oplus_p(-B)|,
        \end{split}
    \end{equation*}
    verifying \eqref{reflection_eq}.

    To prove the other direction, suppose that \eqref{reflection_eq} holds, and assume that for each $i\in [n]$ either $|P_{\R e_i}A|_1>0$, or $|P_{\R e_i}B|_1>0$. We will show that for each $i\in [n]$ either $P_{\R e_i}A=\{0\}$, or $P_{\R e_i}B=\{0\}$. For contradiction, assume that there exists $i\in [n]$ such that $|P_{\R e_i}A|_1>0$, and $|P_{\R e_i}B|_1>0$. Choose $\alpha_i,\beta_i>0$ such that $P_{\R e_i}A=[0,\alpha_ie_i]$ and $P_{\R e_i}B=[0,\beta_ie_i]$. We claim that 
    \begin{equation}\label{single_sum}
        |P_{\R e_i}S_i(A\oplus_p-B) |_1=\|(\alpha_i,\beta_i)\|_1,
    \end{equation}
    and
    \begin{equation}\label{split_sum}
        |P_{\R e_i}(S_i(A)\oplus_p S_i(-B))|_1=\|(\alpha_i,\beta_i)\|_p.
    \end{equation}
    One can verify \eqref{single_sum} using that
    \begin{equation*}
        \begin{split}
            P_{\R e_i}A\oplus_pP_{\R e_i}(-B)&=\bigcup_{t\in [0,1]}[-(1-t)^{\frac{1}{q}}\beta_ie_i,t^{\frac{1}{q}}\alpha_ie_i]=[-\beta_ie_i,\alpha_ie_i],
        \end{split}
    \end{equation*}
    and that
    \begin{equation*}
        |P_{\R e_i}S_i(A\oplus_p-B) |_1=|S_i(P_{\R e_i}(A\oplus_p -B))|_1=|S_i(P_{\R e_i}A\oplus_p P_{\R e_i}(-B))|_1.
    \end{equation*}
    To verify \eqref{split_sum}, use that 
    \begin{equation*}
        P_{\R e_i}S_i(A)\oplus_p P_{\R e_i}S_i(-B)=\bigcup_{t\in [0,1]}[0,t^{\frac{1}{q}}\alpha_i+(1-t)^{\frac{1}{q}}\beta_ie_i]=[0,\|(\alpha_i,\beta_i)\|_p e_i].
    \end{equation*}
    From \eqref{single_sum}, \eqref{split_sum}, and the fact that the quantity $\|(\alpha_i,\beta_i)\|_p$ is strictly decreasing in $p$, we conclude that
    \begin{equation*}
       |P_{\R e_i}(S_i(A)\oplus_p S_i(-B))|_1<|P_{\R e_i}S_i(A\oplus_p-B) |_1. 
    \end{equation*}
    Then $S_i(A)\oplus_p S_i(-B)$ is properly contained in $S_i(A\oplus_p -B)$.
    By Lemma \ref{strict_containment_implies_strict_bound}, we have a contradiction to \eqref{reflection_eq}. 
    Then, for each $i\in [n]$ exactly one of $P_{\R e_i}A$ and $P_{\R e_i}B$ is $\{0\}$. Choose $I_A\subset [n]$ so that $\textup{aff}(A)=\textup{span}\{e_i\}_{i\in I_A}$, and choose $I_B\subset [n]$ so that $\textup{aff}(B)=\{e_i\}_{i\in I_B}$. We just showed that $I_A\cap I_B=\varnothing$, which is equivalent to $\textup{aff}(A)\cap \textup{aff}(B)=\{0\}$.
\end{proof}



Next, we prove Theorem \ref{thm:anti-blocking_eq}. 
For completeness, we reproduce the argument of \cite{MNZ-25} showing the inequality in our Proposition~\ref{RS-l-anti} in the appendix.

\begin{proof}[Proof of Theorem \ref{thm:anti-blocking_eq}]
  Let $K$ be a simplex of the form $K=\textup{conv}(0,\alpha_1e_1,\dots,\alpha_ne_n)$ for real numbers $\alpha_i\neq0$. Then there exists a linear image of $K$ which is an anti-blocking simplex. Hence, it satisfies the equality case as proved in \cite[Lemma 23]{MNZ-25}.

    We now prove the other direction. Let $K$ be a locally anti-blocking convex body satisfying the equality case of Theorem~\ref{thm:anti-blocking_eq}. Since $K$ is full-dimensional, there exists $\delta \in \{-1,1\}^n$ such that $K_\delta$ is a convex body where $K_\delta := \{x \in K: \delta_ix_i \geq 0\}$. By applying a linear transformation, we may assume without loss of generality that $\delta=(1,\dots,1)$. It follows from the proof of inequality recalled in the appendix that $K$ must satisfy the following condition, corresponding to the equality cases in \eqref{eq:k-delta}: for each $\delta\in\{-1,1\}^n$, one has
        \begin{equation}\label{section_eq}
            |K_{\delta}\oplus_p-(K_{-\delta})|=|K_{\delta}\oplus_pK_{-\delta}|.
        \end{equation}
    Using Lemma \ref{reflection_lemma}, we obtain that $\textup{aff} (K_{-\delta}) = \{0\}$ and hence $K$ is anti-blocking.
    This completes the proof, since equality among anti-blocking convex bodies is attained only for simplices.
\end{proof}

\section{Proof of Theorem \ref{thm:Planar-case-general}} \label{subsec:r2}
Let $e_i^{+} := \{x: \langle x, e_i\rangle \geq 0\}$ and $e_i^{-} := \{x: \langle x, e_i\rangle \leq 0\}$ denote the closed half-spaces that are orthogonal to $e_i$. If $K\subset\mathbb{R}^2$ is a convex set, define $K^+:=K\cap e_2^+$ and $K^{-}:=K\cap e_2^-$ to be the ``upper" and ``lower" parts of $K$ with respect to the line $\{x:\langle x,e_2\rangle =0\}$.
\begin{lemma}
\label{lem:upper-lower-part-lp}
    Fix $p\in (1,\infty)$ and let $A,B\subset\mathbb{R}^2$ be convex sets such that $P_{e_2^\perp} A =  A \cap e_2^\perp $ and $ P_{e_2^\perp} B =  B \cap e_2^\perp$. Then, 
    \begin{equation}
        (A \oplus_p B)^+ = A^+ \oplus_p B^+
        \quad\text{and}\quad (A \oplus_p B)^- = A ^- \oplus_p B^- .
    \end{equation}
\end{lemma}

\begin{proof}
    
    We will concentrate on the first relation. Once we finish the proof, we can prove the second relation by interchanging the roles of $+$ and $-$. Let $z=(z_1,z_2) \in A^+ \oplus_p B^+$. We use \eqref{def:Firey-sums-non-convex} to find
    $u \in [0,1]$, $(x,s) \in A^+$, and $(y,t) \in B^+$ such that
    $
        z = u^{1/q} (x,s) + (1-u)^{1/q} (y,t).
    $
    Since $s$ and $t$ are nonnegative, we have $z_2  = u^{1/q} s + (1-u)^{1/q} t 
    \geq 0$, so that $z \in (A\oplus_p B)^+$. That is, we have the containment $(A\oplus _p B)^+\supset A^+\oplus_p B^+$.

    To prove the opposite containment, let $z=(z_1,z_2) \in (A \oplus_p B)^+$. Then $z_2 \geq 0$, and it follows from \eqref{def:Firey-sums-non-convex} that $z = u^{1/q} (x,s) + (1-u)^{1/q} (y,t)$ for $u \in [0,1]$, $(x,s) \in A$, and $(y,t) \in B$. We consider the following cases. First, if $u =0$, then $ z \in A^+ \subset A^+ \oplus_p B^+$ and when $u = 1$, $ z \in B^+\subset A^+ \oplus_p B^+$. We assume that $ u \in (0,1)$.
    \begin{itemize}
        \item[Case 1:] If $s\ge0$ and $t\ge0$, then  $z \in A^+\oplus_p B^+$, which is what we want.
        \item[Case 2:] If $s\le0$ and $t\le0$, then 
        $ z_2 = u^{1/q} s + (1-u)^{1/q} t = 0$,
        and so we have $ s= t=0$, which implies $z\in A^+\oplus _p B^+$.
        \item[Case 3:] If $s\ge0$ and $t\le0$,
         set $s' = u^{-1/q}z_2$. We get that $0\leq s'\leq s$. The assumption $P_{e_2^{\bot}}A=A\cap e_2^{\bot}$ and the convexity of $A^+$ imply that the interval $\{x\}\times [0,s]$ is contained in $A^{+}$, and therefore that $(x,s')\in A^{+}$. Then
        $
            z 
            = u^{1/q}(x,s')+ (1-u)^{1/q} (y,0) \in A^+ \oplus_p B^+.
        $
    \end{itemize}
    We do not need to consider a fourth case because if $(x,s)\in A^-$ and $(y,t)\in B^+$, we can use the same argument as case 3 to show that $z\in A^{+}\oplus_p B^{+}$.
    
    The above cases establish the containment $(A\oplus_p B)^+\subset A^{+}\oplus_p B^+$.
\end{proof}

Using Lemma \ref{lem:upper-lower-part-lp}, we will prove the following generalization of  \eqref{eq:rogers_shephard_conjecture} 
in the case $n=2$.

\begin{prop} \label{prop:lp-two-bodies}
    Fix $p\in (1,\infty)$, and let $A,B\subset\mathbb{R}^2$ be convex sets such that 
    \begin{equation}
    \label{assmp:lp-two-bodies}
      P_{e_2^{\perp}} A =  A \cap e_2^\perp = P_{e_2^{\perp}} B =  B \cap e_2^\perp =[-c,d]  
    \end{equation}
    for some $c,d\ge0$. Then,
    \begin{equation}
        \label{eq:lp-two-bodies}
        |A\oplus_p - B| \leq \frac{\kappa_{2,p}}{2}\left(|A| +|B|\right).
    \end{equation}
\end{prop}

\begin{proof}
Use Lemma \ref{lem:upper-lower-part-lp} and the fact that the union $A\oplus_p -B=(A\oplus_p -B)^+\cup (A\oplus_p -B)^-$ is disjoint almost everywhere to get
    \begin{align}
        |A \oplus_p -B| &= |(A \oplus_p -B)^+| + |(A \oplus_p -B)^-|
        = |A^+ \oplus_p (-B)^+| + |A^- \oplus_p (-B)^-|.
    \end{align}
    Notice that $(-B)^+ = -(B^-)$ and $(-B)^- = -(B^+)$.
If we set $K := A^+ \cup B^-$,
then it follows from assumption \eqref{assmp:lp-two-bodies} that $K$ is a convex set containing the origin. Use Lemma \ref{lem:upper-lower-part-lp} to compute
    \begin{equation}
    \begin{split}
        |K \oplus_p -K| &= |(K \oplus_p -K)^+| + |(K \oplus_p -K)^-| =|K^+ \oplus_p (-K)^+| + |K^- \oplus_p (-K)^-| 
        \\
        &= |A^+ \oplus_p -(B^-)| + | B^- \oplus_p -(A^+)| = 2|A^+ \oplus_p -(B^-)|. 
    \end{split}
    \label{eq:lp-two-bodies-003}
    \end{equation}
    Next, we use the inequality \eqref{eq:rogers_shephard_conjecture}
    to obtain
    \begin{equation}
        \label{eq:lp-two-bodies-001}
        2|A^+ \oplus_p -(B^-)| 
        = |K \oplus_p -K | 
        \leq
        \kappa_{2,p} |K| = \kappa_{2,p} (|A^+|+|B^-|).
    \end{equation}
    By a similar argument, we find that 
    \begin{equation}
        \label{eq:lp-two-bodies-002}
        2|A^- \oplus_p -(B^+)| \leq \kappa_{2,p} (|A^-|+|B^+|).
    \end{equation}
    Finally, we combine \eqref{eq:lp-two-bodies-001} and \eqref{eq:lp-two-bodies-002} to conclude that
    \[
        |A\oplus_p -B|\leq \frac{\kappa_{2,p}}{2}((|A^+|+|A^-|)+(|B^-|+|B^+|))=\frac{\kappa_{2,p}}{2}(|A|+|B|). \qedhere
    \]
\end{proof}

\begin{lemma}\label{lem:decreasing-half-equality}
    Fix $p\in (1,\infty)$, and let $A,B\subset \mathbb{R}^2$ be two convex sets of the form 
    \begin{equation}
        A = \{ (x,s): -f(x) \leq s \leq f(x), x \in [0,c]\},\quad  B = \{ (x,s):  -g(x) \leq s \leq g(x),x \in [0,c]\},
    \end{equation}
    where $f$ and $g$ are concave functions on $[0,c]$ and $c>0$. 
  If the equality
  \begin{equation}
        \label{eq:decreasing-half-equality}
      |A\oplus_p -B| = \frac{\kappa_{2,p}}{2}\left(|A| +|B|\right),
  \end{equation} 
  holds, then we have 
  \begin{equation}
    \label{eq:decreasing-half-equality-2}
  |A^+\oplus_p -(A^+)| = \kappa_{2,p}|A^+|\quad\hbox{and}\quad|B^+\oplus_p -(B^+)| = \kappa_{2,p}|B^+|.
  \end{equation}
  Moreover, under the equality assumption in \eqref{eq:decreasing-half-equality}, if $f$ and $g$ are not identically zero, then $f$ and $g$ are non-decreasing on $[0,c]$.
\end{lemma}

\begin{proof} 
    We can assume that $c=1$, applying a dilation to the sets $A$ and $B$ if necessary. Let the sets
    \[
        A_t = \{ (x,s) : -(1-t) f(x) \leq s \leq (1+t)f(x) , x\in [0,1]\}
    \]
    and 
    \[
        B_t = \{ (y,s) : -(1-t) g(y) \leq s \leq (1+t)g(y), y\in [0,1] \}
    \]
    define parallel chord movements of $A$ and $B$ along the direction $e_2$ with parameter $t\in [-1,1]$. From Theorem \ref{thm:BC-shadow-lp} we know that $\{A_t \oplus_p -B_t\}_{t \in [-1,1]}$ is also a parallel chord movement along the direction $e_2$. We use the additional facts that $|A_t| = |A|$ and $|B_t|=|B|$ for any $t\in[-1,1]$ to deduce that the function
    $$
        R(t):= \frac{|A_t \oplus_p -B_t|}{|A_t|+|B_t|}
    $$ 
    is convex and even (by symmetry of $A$ and $B$) on the interval $[-1,1]$.  It follows that $R$ achieves its minimum at $t=0$ and so $R(0)\leq R(t)$ for all $t\in [-1,1]$. 
    On the other hand, since $A_t$ and $B_t$ satisfy \eqref{assmp:lp-two-bodies}, we can apply  \eqref{eq:lp-two-bodies} to get 
    \[
        R(0)= \frac{|A\oplus_p -B|}{|A| +|B|} = \frac{\kappa_{2,p}}{2} \geq \frac{|A_t \oplus_p -B_t|}{|A_t|+|B_t|} = R(t).
    \]
  We deduce that $R$ is constant on $[-1,1]$. 
    Now, use \eqref{eq:rogers_shephard_conjecture} and Lemma \ref{lem:upper-lower-part-lp} to compute 
    \begin{align}
        \frac{\kappa_{2,p}}{2} |A| &= \kappa_{2,p}|A^+|\geq
        |A^+ \oplus_p -(A^+)| 
        = |(A^+ \oplus_p -(A^+))^+| + |(A^+ \oplus_p -(A^+))^-|
        \\
        &= |A^+ \oplus_p [-e_1,0]| + |[0,e_1]\oplus_p -(A^+)| 
        =2|A^+ \oplus_p [-e_1,0]| .
    \end{align}
    Let $T:\R^2\to\R^2$ be the linear map defined by $T(e_1)=e_1$ and $T(e_2)=2e_2$. Then $A_1=T(A_+)$ and thus 
    \[
       \frac{\kappa_{2,p}}{2} |A|  \geq |A_1 \oplus_p [-e_1,0]| = |A_1^+ \oplus_p (-B_1)^+|   = |(A_1 \oplus_p -B_1)^+|.
    \]
    By the same argument, we obtain
    \begin{align}
        &\frac{\kappa_{2,p}}{2} |B| 
        \geq |(A_1 \oplus_p -B_1)^-| .
    \end{align}
    Using Lemma \ref{lem:upper-lower-part-lp} and the fact that $R(t)=\frac{\kappa_{2,p}}{2}$ for all $t\in [-1,1]$, we obtain
    \begin{align}
        \frac{\kappa_{2,p}}{2} (|A|+|B|) 
        &\geq |(A_1 \oplus_p -B_1)^+| + |(A_1 \oplus_p -B_1)^-| = |A_1 \oplus_p -B_1| 
        \\
        &= (|A_1|+|B_1|)R(1) 
        =\frac{\kappa_{2,p}}{2} (|A|+|B|).
    \end{align}
    Hence, there is equality in all of the above inequalities, and \eqref{eq:decreasing-half-equality-2} holds.

    Next, assume $f$ and $g$ are not identically zero. We will show that $f$ is non-decreasing. First, notice that the function defined by
    $$
    H_z(t):= |(z+\mathbb{R}e_2) \cap (A_t \oplus_p -B_t)|_1,\qquad z\in[-1,1] 
    $$ 
    is a convex function of $t\in[-1,1]$. Then we have
    \begin{equation*}
        \frac{H_z(-1)+H_z(1)}{2} \geq H_z(0),
    \end{equation*}
    and by the symmetry of  $A$ and $B$, we find that $H_z$ is even. 
    We use Fubini's theorem and the fact that $R(t)$ is constant to obtain 
    \[
    \int_{-1}^1 H_z(0) dz \leq \int_{-1}^1 \frac{H_z(-1)+H_z(1)}{2} dz =\frac{R(-1)+R(1)}{2} = R(0)=\int_{-1}^1 H_z(0) dz,
    \]
    and therefore 
    $H_z (t)$ is constant on $[-1,1]$ for almost all $z\in[-1,1]$. Since the function $z\mapsto H_z(t)$ is concave on $[-1,1]$, it is continuous on $[-1,1]$ and we conclude that $t\mapsto H_z(t)$ is constant for all $z\in[-1,1]$.
    Using Lemma \ref{lem:upper-lower-part-lp}, equation \eqref{def:Firey-sums-non-convex}, and the fact that $A \oplus_p B$ is symmetric with respect to $e_2^\perp$, we get
    \begin{equation}
    \begin{split}
        H_0(0) &
       = 2|e_1^\perp\cap (A \oplus_p -B)^+| = 2|e_1
       ^\perp \cap( A^+ \oplus_p (-B)^+)|
       \\
       &=
        2 \max \limits_{\substack{0=u^{\frac{1}{q}}x+(1-u)^{\frac{1}{q}}y}}u^{\frac{1}{q}}f(x)+(1-u)^{\frac{1}{q}}g(-y) ,
    \end{split}
    \label{eq:H-t-zero}
    \end{equation}
    where the maximum is taken over those $x,u\in[0,1]$ and $y\in[-1,0]$. For convenience, we will omit writing the constraints on $x$, $u$, and $y$ in the following computations, and understand that the maximum is taken in this restricted sense. Similarly, we have
    \begin{equation}
    \begin{split}
          H_0(1)&= |e_1^\perp\cap (A_1\oplus_p -B_1)^+| +|e_1^\perp \cap (A_1 \oplus_p -B_1)^-|
          \\
          &= |e_1^\perp\cap (A_1\oplus_p [-e_1,0])| +|e_1^\perp \cap ([0,e_1] \oplus_p -B_1)|
          \\
          &=
          \max_{0=u^\frac{1}{q}x+(1-u)^\frac{1}{q}y}u^\frac{1}{q}2 f(x) + \max_{0=u^\frac{1}{q}x+(1-u)^\frac{1}{q}y}(1-u)^\frac{1}{q}2 g(-y).
    \end{split}
    \label{eq:H-t-one}
    \end{equation}
    Let $u_0\in [0,1]$, $x_0\in [0,1]$ and $y_0 \in [-1,0]$ denote the points that satisfy $u_0^{1/q} x_0 = -(1-u_0)^{1/q} y_0$, where  the maximum in \eqref{eq:H-t-zero} is attained. By the triangle inequality, we get  
    \begin{align}
        H_0(0)&= 2 \left(u_0^\frac{1}{q} f(x_0) + (1-u_0)^\frac{1}{q} g(-y_0)\right)
        \\
        &\leq \max_{0=u^\frac{1}{q}x+(1-u)^\frac{1}{q}y}u^\frac{1}{q}2f(x) + \max_{0=u^\frac{1}{q}x+(1-u)^\frac{1}{q}y}(1-u)^\frac{1}{q}2g(-y) =H_0(1),
    \end{align}
    where we used \eqref{eq:H-t-one} in the last step.
    Thus,
    $$
        u_0^{\frac{1}{q}} f(x_0)=\max_{0=u^{\frac{1}{q}}x+(1-u)^{\frac{1}{q}}y}u^{\frac{1}{q}}f(x) \text{ and }(1-u_0)^{\frac{1}{q}} g(-y_0) = \max_{0=u^{\frac{1}{q}}x+(1-u)^{\frac{1}{q}}y}(1-u)^{\frac{1}{q}}g(-y).
    $$
    Note that $f(x_0), g(-y_0) \neq 0$. Otherwise, the fact $f(x_0) =0$ would imply that $f \equiv 0$, and the fact $g(-y_0) =0$ would imply that $g \equiv 0$, which contradicts our assumption that $f$ and $g$ are not identically zero.
    By considering $u = \frac{1}{1+x^q}$ and $y= -1$, we have
    \begin{align}
     \label{eq:decreasing-half-equality-001}
         u_0^{\frac{1}{q}} f(x_0) 
        \geq \max_{x} \frac{f(x)}{(1+x^q)^{\frac{1}{q}}} 
        \geq \frac{f(x_0)}{(1+x_0^q)^{\frac{1}{q}}},
    \end{align}
    and so $u_0 \geq \frac{1}{1+x_0^q}$. On the other hand, considering $u = \frac{(-y_0)^q}{1+(-y_0)^q}, x =1$, and $y = y_0$, we get
    \begin{align}
        (1-u_0)^{\frac{1}{q}} g(-y_0) 
        \geq \frac{g(-y_0)}{(1+(-y_0)^q)^{\frac{1}{q}}},
    \end{align}
    and then $u_0 \leq \frac{(-y_0)^q}{1+(-y_0)^q}$. We have that
    \[
        \frac{1}{1+x_0^q} \leq u_0 \leq \frac{(-y_0)^q}{1+(-y_0)^q}, 
    \]
    which implies that $1\leq x_0^q(-y_0)^q$. It follows that $ x_0 = 1 = -y_0$ and that $u_0 =1/2$. Using \eqref{eq:decreasing-half-equality-001}, we find that for any $x \in [0,1]$,
    \[
        \frac{f(x)}{(1+x^q)^{\frac{1}{q}}} \leq \frac{f(1)}{2^{\frac{1}{q}}},
    \]
    so that
    \begin{align*}
        f(x) \leq f(1) \left(\frac{1+x^q}{2}\right)^{\frac{1}{q}} \leq f(1).
    \end{align*}
    As a result, because $f$ is also a concave function, $f$ is non-decreasing. Note also that $g$ is non-decreasing. Indeed, this follows by replacing $A$ by $B$ and $B$ by $A$.
\end{proof}

    \begin{lemma}
        \label{lem:formula-h+}
        Let $K \subset \R^2_+ $ be a convex body of the form 
        \[
            K = \{(x,s):0\leq s \leq f(x), x \in [0,1] \},
        \]
        for some concave function $f$. For $z\in[-1,1]$, denote
        \[
            H^{+}(z,K) = | (z+\mathbb{R}e_2) \cap (K \oplus_p -K)^+|.
        \]
        Let $x_0\in [0,1]$ be a point such that $f(x_0) = \|f\|_\infty$ and let $\Phi_z$ be a function on $[0,1]$ such that $ \Phi_z : u \mapsto \frac{(1-u)^{1/q}+z}{u^{1/q}} $. Then,
        \begin{equation}
            H^{+}(z,K)
            =
            \begin{cases}
                0, \quad &z = -1,
                \\
                \displaystyle 
                \max_{0\leq x \leq x_0} (\Phi_z^{-1}(x))^{1/q} f(x), \quad &z \in (-1,0),
                \\
                \displaystyle 
                \max_{z\leq x \leq x_0} (\Phi_z^{-1}(x))^{1/q} f(x), \quad &z \in [0,x_0],
                \\
                f(z), \quad &z\geq x_0.
            \end{cases}
            \label{eq:formula-h+}
        \end{equation}
    \end{lemma}

    \begin{proof}
        Since $f$ is concave and its maximum is attained at $x_0$, $f$ is non-decreasing on $[0,x_0]$ and non-increasing $[x_0, 1]$. 
        From Lemma \ref{lem:upper-lower-part-lp}, we have 
\begin{equation}
\begin{split}
    H^{+}(z,K)
    &= | (z+\mathbb{R}e_2) \cap (K^+ \oplus_p (-K)^+)| 
    \\
    &= | (z+\mathbb{R}e_2) \cap (K \oplus_p [-e_1,0])| 
    =
    \max_{z = u^{\frac{1}{q}}x + (1-u)^{\frac{1}{q}}y} u^{\frac{1}{q}} f(x), 
\end{split}
\label{eq:arg-h+}
\end{equation}
where the maximum is taken over all $u\in[0,1], x \in [0,1] $ and $y \in [-1,0]$. Just as we did in the previous lemma, we will omit writing the constraints on $x$, $y$, and $u$, and understand that the maximum is taken in this restricted sense. 

First, consider the case where $z \geq x_0$. Since $K \subset K \oplus_p [-e_1,0]$, we have $f(z) \leq H^{+}(z,K)$. If we take any $u,x\in[0,1]$ and $y\in [-1,0]$ that satisfy $z = u^{1/q}x + (1-u)^{1/q}y$, then because $f$ is a concave function we get
\begin{align*}
    (1-u^\frac{1}{q})f(0)+u^\frac{1}{q}f(x)\leq f(u^\frac{1}{q}x),
\end{align*}
and therefore
\[
    u^{\frac{1}{q}} f(x) \leq  f(u^{\frac{1}{q}}x) \leq f(z),
\]
where the second inequality in the above line follows from the facts that $u^{1/q}x \geq z$ and $f$ is non-increasing on $[x_0,1]$. Thus, we have $H^+(z,K)\leq f(z)$, and because the inequality holds in the other direction as well, we have $H^{+}(z,K) = f(z)$ for $z \geq x_0$.

Let $z \in [0,x_0]$. The function $\Phi_z : [0,1] \rightarrow [z,\infty]$, defined by
$
    \Phi_z : u \mapsto \frac{(1-u)^{1/q}+z}{u^{1/q}},
$ 
is a bijection that is decreasing in $u$. For $u\in [0,1]$, define
$$
    F(u) := \max_{x \in \left[\frac{z}{u^{1/q}}, \Phi_z(u) \right]} f(x).
$$
Notice that in the argument of the maximum in \eqref{eq:arg-h+} we have $z \leq u^{1/q}x \leq u^{1/q}$, from which we get $u\geq z^q$.
Thus,
\begin{equation}
\begin{split}
    H^{+}(z,K) &=\max_{z^q\leq u}u^{\frac{1}{q}}\max_{0\leq \frac{u^{1/q}x-z}{(1-u)^{1/q}}\leq 1}f(x)\\
    &=\max_{z^q\leq u}u^{\frac{1}{q}}\max_{\frac{z}{u^{1/q}}\leq x\leq 
    \Phi_z(u)
    }f(x)= \max_{z^q \leq u } u^{\frac{1}{q}} F(u)
    \\
    &=\max \left( 
    \max_{z^q \leq u \leq (\frac{z}{x_0})^q} u^{\frac{1}{q}} F(u)
    ,
    \max_{(\frac{z}{x_0})^q \leq u \leq \Phi_z^{-1}(x_0)} u^{\frac{1}{q}} F(u)
    ,
    \max_{\Phi_z^{-1}(x_0) \leq u } u^{\frac{1}{q}} F(u)
    \right).
\end{split}
\label{eq:possible-3-cases}
\end{equation}
Note that since $z \leq z + (1-(\frac{z}{x_0})^q)^{1/q}$, we obtain 
\begin{align}
    x_0 \leq \frac{z + (1-(\frac{z}{x_0})^q)^{1/q}}{\frac{z}{x_0}} = \Phi_z \left(\left(\frac{z}{x_0}\right)^q \right) 
    \Rightarrow \Phi_z^{-1} (x_0)\geq \left(\frac{z}{x_0}\right)^q .
    \label{eq:left-case-002}
\end{align}
Consider first when $z^q \leq u \leq (\frac{z}{x_0})^q$. We have $x_0 \leq \frac{z}{u^{1/q}}$. Since $f$ is non-increasing on $\left[x_0, 1\right]$, we obtain
\begin{align}
    \max_{z^q \leq u \leq (\frac{z}{x_0})^q} u^{\frac{1}{q}} F(u) = \max_{z^q \leq u \leq (\frac{z}{x_0})^q} u^{\frac{1}{q}} f\left(\frac{z}{u^{1/q}}\right) = \left(\frac{z}{x_0}\right) f(x_0)\leq (\Phi_z^{-1} (x_0))^{\frac{1}{q}} f(x_0),
    \label{eq:left-case}
\end{align}
where we used \eqref{eq:left-case-002} in the last inequality.
Now if $(\frac{z}{x_0})^q \leq u \leq \Phi_z^{-1}(x_0)$, we have $x_0 \in \left[\frac{z}{u^{1/q}}, \Phi_z(u) \right]$. Thus,
\begin{align}
    \max_{(\frac{z}{x_0})^q \leq u \leq \Phi_z^{-1}(x_0)} u^{1/q} F(u) = \max_{(\frac{z}{x_0})^q \leq u \leq \Phi_z^{-1}(x_0)} u^{1/q} f(x_0) = \left(\Phi_z^{-1}(x_0) \right)^{1/q} f(x_0).
    \label{eq:middle-case}
\end{align}
Consider $\Phi_z^{-1}(x_0) \leq u$. Thus, $x_0 \geq \Phi_z (u)$. Since $f$ is non-decreasing on $[0,x_0]$, we obtain
\begin{align}
    \max_{\Phi_z^{-1}(x_0) \leq u } u^{1/q} F(u) =\max_{\Phi_z^{-1}(x_0) \leq u } u^{1/q} f(\Phi_z (u)) = \max_{z\leq x \leq x_0} (\Phi_z^{-1}(x))^{1/q} f(x).
    \label{eq:last-case}
\end{align}
Hence, using \eqref{eq:possible-3-cases}, \eqref{eq:left-case}, \eqref{eq:middle-case} and \eqref{eq:last-case},
\[
    H^{+}(z,K) = \max_{z\leq x \leq x_0} (\Phi_z^{-1}(x))^{1/q} f(x), \quad z \in [0,x_0].
\]

Now, consider $z \in (-1,0)$. The function $\Phi_z : [0,1] \rightarrow [z,\infty]$ defined by
$
    \Phi_z : u \mapsto \frac{(1-u)^{1/q}+z}{u^{1/q}}, 
$ 
is a bijection and decreasing in $u$. Denote
\[
F(u) = \max_{x\in[0,\Phi_z(u)]} f(x).
\]
Notice that the argument in the maximum of \eqref{eq:arg-h+}, $u\leq 1- (-z)^q$ since $z \geq (1-u)^{1/q}y \geq -(1-u)^{1/q}$. Thus,
\begin{equation}
    \begin{split}
        H^{+}(z,K) &=\max_{u\leq 1-(-z)^q} u^{1/q} F(u)
        \\
        &=\max \left( 
    \max_{0\leq u \leq \Phi_z^{-1}(x_0)} u^{1/q} F(u)
    ,
    \max_{\Phi_z^{-1}(x_0) \leq u \leq 1-(-z)^q} u^{1/q} F(u)
    \right).
    \end{split}
    \label{eq:possible-2-cases}
\end{equation}
Consider when $u \leq \Phi_z^{-1} (x_0)$ and so $\Phi_z (u) \geq x_0$. Thus,
\begin{align}
     \max_{0\leq u \leq \Phi_z^{-1}(x_0)} u^{1/q} F(u) =  \max_{0\leq u \leq \Phi_z^{-1}(x_0)} u^{1/q} f(x_0) = (\Phi_z^{-1}(x_0))^{1/q} f(x_0).
     \label{eq:left-case-2}
\end{align}
When $\Phi_z^{-1}(x_0) \leq u \leq 1-(-z)^q$, we have $x_0 \geq \Phi_z (u) \geq \Phi_z( 1- (-z)^q) = 0$. Thus, $f$ is non-decreasing,
\begin{equation}
     \max_{\Phi_z^{-1}(x_0) \leq u \leq 1-(-z)^q} u^{1/q} F(u) = \max_{\Phi_z^{-1}(x_0) \leq u \leq 1-(-z)^q} u^{1/q} f(\Phi_z(u)) = \max_{0\leq x \leq x_0} (\Phi_z^{-1}(x))^{1/q} f(x).
     \label{eq:right-cases-2}
\end{equation}
Hence, using \eqref{eq:possible-2-cases}, \eqref{eq:left-case-2} and \eqref{eq:right-cases-2},
\[
    H^{+}(z,K) = \max_{0\leq x \leq x_0} (\Phi_z^{-1}(x))^{1/q} f(x), \quad z \in (-1,0).
\]
To conclude, when $z= -1$, 
then it forces $u = 0 =x$ and $y = -1$, and so $H^{+}(-1,K) = 0$. 
    \end{proof}


    We recall that a chord $[a, b]$ of a convex body $K \subset \R^n$ is said to be an affine diameter of $K$, if there are two parallel, distinct hyperplanes $H_a$ and $H_b$ both supporting $K$ such that $a \in H_a$ and $b \in H_b$. It was proved in \cite{H-51,H-63} and also in \cite[Fact 3.3]{S-05} that any point of a convex body $K$ belongs to some affine diameter of $K$.
   

\begin{proof}[Proof of Theorem \ref{thm:Planar-case-general}]
    Let $K$ be a convex body attaining equality in \eqref{eq:rogers_shephard_conjecture}. There exist $u\in\mathbb{S}^{1}$ and $c,d\geq0$ such that the chord $[-cu,du]\subset K$ is an affine diameter, thus there exist two supporting lines of $K$ at the points $-cu$ and $du$ that are distinct and parallel. Let $v\in\mathbb{S}^{1}$ be a vector parallel to these supporting lines. Applying a linear map $T$ such that $T u=e_1$ and $T v=e_2$, we obtain that the image of the affine diameter is on the $x$-axis and the corresponding supporting lines of $TK$ are parallel to $e_2$. In particular, we have
    \[
        P_{e_2^\perp}(TK) = (TK)\cap e_2^\perp.
    \]
    Using \eqref{eq:linear-invariant-lp}, it suffices to prove that $TK$ is a triangle with one vertex at the origin. Without loss of generality, we may assume that $K $ satisfies $P_{e_2^\perp}K = K\cap e_2^\perp = [-c,d]$.

    Observe that $K_t = K+te_1$ is a parallel chord movement along the direction $e_1$ on $[-d,c]$. Using Theorem \ref{thm:BC-shadow-lp}, the family $K_t \oplus_p -K_t$ on $[-d,c]$ is also a parallel chord movement, and therefore 
    $$
        R:t \mapsto \frac{|K_t \oplus_p -K_t|}{|K_t|}
    $$ 
    is a convex function in $t \in [-d,c]$ since $|K_t| =|K|$ for any $t.$ Thus the maximum is reached at one of the endpoints $t = c,-d$. In either case, the origin is in the boundary, say the maximum reach at $t = c$. We claim that $K_c$ has to be a triangle having a vertex at the origin. Let us assume for now that this claim holds, and let us clarify what this implies for $K$. First, $K$ must be a triangle.
    Assume that the origin is not a vertex of $K$. Along the parallel chord movement, there are two possible situations, see the pictures below:
    \begin{itemize}
        \item[1.] The origin lies in the relative interior of an edge of $K$.
        \item[2.] The origin lies in the interior of $K$ and in this case, $K_{-d}$ is a simplex whose one vertex lies in the relative interior of an edge.
\end{itemize}

    In the first case, after applying some linear transformation $T$, we have that $TK$ is a locally
anti-blocking. It then follows from Theorem~\ref{thm:anti-blocking_eq} that $TK$ must be a triangle with one
vertex at the origin, and hence the same holds for $K$. In the second case, the origin lies in the interior of $K$. So the convex function $R$ reaches its maximum at zero, which is in the interior of $[-d,c]$. It implies that $R$ is constant on $[-d,c]$ and thus that $R(-d)=R(c)$, which is not possible by the equality case of Theorem~\ref{thm:anti-blocking_eq}, since $K_{-d}$ is locally anti-blocking.
      
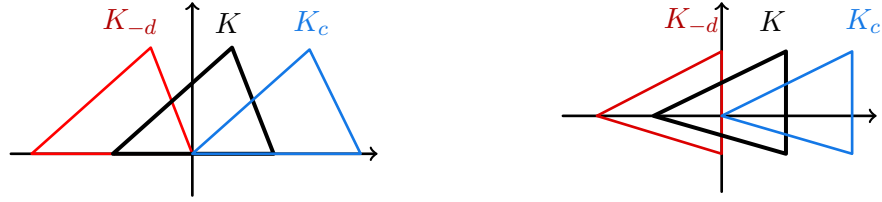
\begin{center}
    \begin{tikzpicture}[
        line cap=round,
        line join=round,
        scale=.5
    ]
    
    \definecolor{mygreen}{RGB}{120,190,50}
    \definecolor{myorange}{RGB}{245,120,20}
    \definecolor{mybrown}{RGB}{80,45,25}
    \definecolor{myblue}{RGB}{20,120,230}
    \definecolor{myred}{RGB}{220,0,0}
    \begin{scope}[xshift=0cm]
    
        \draw[->,line width=1pt] (-4.8,0) -- (4.9,0);
        \draw[->,line width=1pt] (0,-1.1) -- (0,4);
    
        \draw[red, line width=1pt]
            (-4.25,0) -- (-1.1,2.8) -- (0,0) -- cycle;
    
        \draw[black, line width=1.5pt]
            (-2.1,0) -- (1.05,2.8) -- (2.15,0) -- cycle;
    
        \draw[myblue, line width=1pt]
            (0,0) -- (3.1,2.75) -- (4.45,0) -- cycle;
    
    
        \node[text=myred!80!black] at (-1.65,3.45) {$K_{-d}$};
        \node[text=black!90!black] at (0.95,3.45) {$K$};
        \node[text=myblue ] at (3.15,3.45) {$K_c$};
    
    \end{scope}
    
    \begin{scope}[xshift=14cm,yshift=1cm]
    
        \draw[->,line width=1pt] (-4.2,0) -- (4.1,0);
        \draw[->,line width=1pt] (0,-2.0) -- (0,3.0);
    
        \draw[myred, line width=1pt]
            (-3.3,0) -- (0,1.7) -- (0,-1) -- cycle;
    
        \draw[black, line width=1.5pt]
            (-1.8,0) -- (1.7,1.7) -- (1.7,-1) -- cycle;
    
        \draw[myblue, line width=1pt]
            (0,0) -- (3.45,1.7) -- (3.45,-1) -- cycle;
    
    
        \node[text=myred!80!black] at (-.8,2.5) {$K_{-d}$};
        \node[text=black!90!black] at (1.35,2.5) {$K$};
        \node[text=myblue] at (3.75,2.5) {$K_c$};
    
    \end{scope}

    \end{tikzpicture}
    \end{center}

    Now, we prove the claim. Replacing $K_c $ by $ K$ and rescaling, we may assume that $P_{e_2^\perp} K = [0,1]$. Write $K = K^+ \cup K^-$. Define the reflection $R:\mathbb{R}^2 \rightarrow\mathbb{R}^2$ by $R(x_1,x_2) = (x_1,-x_2)$.
    Denote $A = K^+ \cup R (K^+),B = K^- \cup R (K^-)$. We have that  $A$ and $B$ are symmetric with respect to $e_2^\perp$, and then so is $A \oplus_p B$. Thus,
    \begin{equation}
    \begin{split}
        |A \oplus_p -B| &
        =2 |(A\oplus_p -B)^+|  = 2 |A^+\oplus_p (-B)^+|
        = 2|K^+ \oplus_p (-K)^+| 
        =  |K \oplus_p -K|, \label{eq:RS-lp-to-two-bodies-001}
    \end{split}
    \end{equation}
     and 
     \begin{equation}
        |A| +|B| 
        = 2|K^+|+2|K^-| = 2|K|. \label{eq:RS-lp-to-two-bodies-002}
     \end{equation}
     Thus, $A$ and $B$ satisfy  \eqref{eq:decreasing-half-equality}, that is,
     \begin{equation}
         |A \oplus_p -B| = |K \oplus_p -K| = \kappa_{2,p} |K| = \frac{\kappa_{2,p}}{2} (|A| + |B|),
         \label{eq:RS-lp-to-two-bodies-004}
     \end{equation}
     where 
     the second equality comes from the assumption that $K$ satisfies the equality in \eqref{eq:rogers_shephard_conjecture}.
     
     Using Lemma \ref{lem:decreasing-half-equality}, we see that $A^+ = K^+$ and $B^+ = K^-$ attain the equality in \eqref{eq:rogers_shephard_conjecture}. We will show that $K^+$ and $K^-$ are triangles having the origin as a vertex. Assuming this, we now explain how we conclude that $K$ is a triangle with one vertex as the origin.

    Suppose, toward a contradiction, that $K$ is not a triangle with a vertex at the origin. The case where $K$ is a triangle with the origin lying in the relative interior of an edge was excluded by above. Then $K$ is a quadrilateral with one vertex at the origin and another one at $e_1$. In this situation, there exists a parallel chord movement $\mathcal{P}_t$ with $t_0<0<t_1$ such that $\mathcal{P}_0=K$, $\mathcal{P}_{t_1}$ is a triangle with one vertex at the origin, and $\mathcal{P}_{t_0}$ is a triangle for which the origin lies in the relative interior of an edge, see the proof in \cite{CCG-99} and the picture below. Using Theorem \ref{thm:BC-shadow-lp}, we obtain that $\mathcal{P}_t \oplus_p -\mathcal{P}_t$ is also a linear parameter system. We get that $|P_t \oplus -P_t|$ is a convex function in $t \in [t_0,t_1]$. Using Theorem~\ref{thm:anti-blocking_eq}, equality in \eqref{eq:rogers_shephard_conjecture} holds for $\mathcal{P}_{t_1}$, while it is strict for $\mathcal{P}_{t_0}$, which is a linear image of a locally anti-blocking convex body whose origin is not a vertex. Hence. Therefore, 
\[
|\mathcal{P}_{t_0}\oplus_p(-\mathcal{P}_{t_0})|
< |\mathcal{P}_0\oplus_p(-\mathcal{P}_0)| =
|\mathcal{P}_{t_1}\oplus_p(-\mathcal{P}_{t_1})|.
\]
This contradicts the convexity of $t\mapsto |\mathcal{P}_t\oplus_p(-\mathcal{P}_t)|$ on $[t_0,t_1]$. Thus, $K$ must be a triangle with a vertex at the origin.
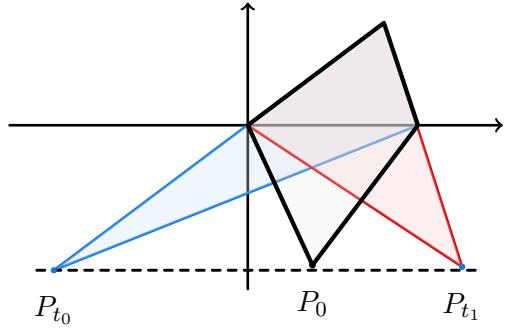
\begin{center}
     \begin{tikzpicture}[scale=.45, line cap=round, line join=round]
\definecolor{myblue}{RGB}{20,120,230}
\definecolor{myred}{RGB}{220,0,0}

\coordinate (O)  at (0,0);
\coordinate (A)  at (4,3);
\coordinate (B)  at (5.0,0);
\coordinate (C)  at (1.9,-4.1);

\coordinate (Pt0) at (-5.7,-4.25);
\coordinate (P0)  at (1.9,-4.1);
\coordinate (Pt1) at (6.3,-4.15);

\draw[->,line width=1pt] (-7.0,0) -- (7.5,0);
\draw[->,line width=1pt] (0,-4.8) -- (0,3.6);

\draw[myblue, line width=1pt] (Pt0)  -- (A) -- (B) -- cycle;
\fill[myblue!20,opacity=.3] (Pt0)  -- (A) -- (B) -- cycle;

\draw[myred, line width=1pt] (O) -- (A) -- (Pt1) -- cycle;
\fill[myred!20,opacity=.3] (O) -- (A) -- (Pt1) -- cycle;

\fill[gray!20,opacity=.2] (O) -- (A) -- (B) -- (C) -- cycle;
\draw[line width=1.6pt] (O) -- (A) -- (B) -- (C) -- cycle;

\draw[dashed, line width=1.1pt] (-6.2,-4.25) -- (6.9,-4.25);

\fill[myblue] (Pt0) circle (2.6pt);
\fill (P0) circle (2.8pt);
\fill[myblue] (Pt1) circle (2.6pt);

\node[below=6pt] at (Pt0) {$P_{t_0}$};
\node[below=6pt] at (P0) {$P_0$};
\node[below=6pt] at (Pt1) {$P_{t_1}$};

\end{tikzpicture}
\end{center}
Now we focus on $K^+$ and so we replace $K^+$ by $K$ for simplification. Write 
\[
    K = \{ (x,s) : 0 \leq s \leq f(x), x\in [0,1] \},
\]
for some concave function $f$. We will start with the case when $f(1) = \|f\|_\infty$.


Assume $f(1) = \|f\|_\infty $, using the concavity of $f$, we have that
$f$ is non-decreasing. Let $l_0$ be a line passing origin and $(1,f(1))$ and $H_0$ and $H_1$ be lines orthogonal to $l_0$ passing through zero and $(1,f(1))$, respectively. Applying a rotation so that $l_0$ is $x-$axis, we denote by $H_0'=y$-axis, $H_1'$ and $l_0'$ for the corresponding rotated lines $H_0,H_1$ and $l_0$, respectively. Let $L$ be the rotated $K.$ Thus, $L$  is between $y$-axis and $H_1'$ such that $P_{e_2^\perp} L = L\cap e_2^\perp$. 
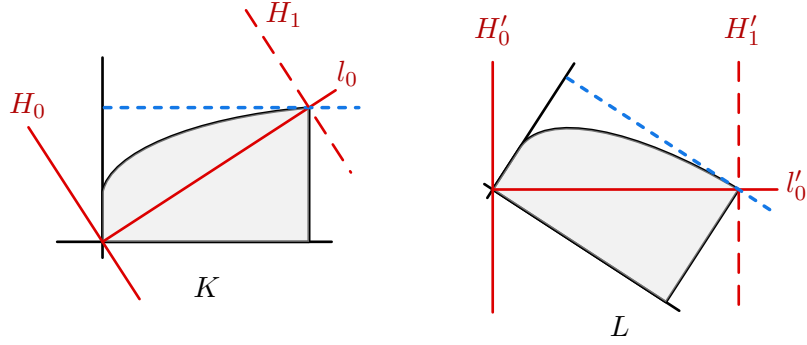
\begin{center}
\begin{tikzpicture}[line cap=round,line join=round,scale=.6]

\definecolor{myred}{RGB}{220,0,0}
\definecolor{myblue}{RGB}{20,120,230}

\pgfmathsetmacro{\Qx}{4.55}   
\pgfmathsetmacro{\Qy}{2.95}
\pgfmathsetmacro{\Py}{1.10}   
\pgfmathsetmacro{\Tx}{5.65}   
\pgfmathsetmacro{\Ty}{2.82}

\pgfmathsetmacro{\ang}{atan2(\Qy,\Qx)} 
\pgfmathsetmacro{\ca}{cos(\ang)}
\pgfmathsetmacro{\sa}{sin(\ang)}
\pgfmathsetmacro{\rho}{veclen(\Qx,\Qy)}

\pgfmathsetmacro{\Lx}{\Py*\sa}          
\pgfmathsetmacro{\Ly}{\Py*\ca}
\pgfmathsetmacro{\Bx}{\Qx*\ca}          
\pgfmathsetmacro{\By}{-\Qx*\sa}

\begin{scope}[shift={(0,0)}]

    \draw[line width=1pt] (-1.00,0) -- (5.05,0);
    \draw[line width=1pt] (0,-0.35) -- (0,4.05);
    \draw[line width=1pt] (\Qx,0) -- (\Qx,\Qy);

    \draw[line width=1pt]
        (0,\Py)
        .. controls (0.18,1.95) and (1.95,2.72) ..
        (\Qx,\Qy)
        -- (\Qx,0) 
        --
        (0,0)
        --
        cycle
        ;
    \fill[gray!20,opacity= .5]
        (0,\Py)
        .. controls (0.18,1.95) and (1.95,2.72) ..
        (\Qx,\Qy)
        -- (\Qx,0) 
        --
        (0,0)
        --
        cycle
        ;

    \draw[myred,line width=1pt]
        (0,0) -- ({(\Qx+0.70*\ca)},{(\Qy+0.70*\sa)});
    \node[text=myred!80!black] at (-1.65,3) {$H_0$};
    \draw[myred,line width=1pt]
        ({3.0*cos(\ang+90)},{3.0*sin(\ang+90)})
        --
        ({1.50*cos(\ang-90)},{1.50*sin(\ang-90)});
    \node[text=myred!80!black] at (5.4,3.7) {$l_0$};
    \draw[myblue,dashed,line width=1.3pt,dash pattern=on 2.5pt off 3.8pt]
        (0.04,\Qy) -- (5.65,\Qy);
    \node[text=myred!80!black] at (4,5) {$H_1$};
    \node[text=black] at (2.3,-1) {$K$};
    \draw[myred,dashed,line width=1pt,dash pattern=on 7pt off 5pt]
        ({\Qx+2.55*cos(\ang+90)},{\Qy+2.55*sin(\ang+90)})
        --
        ({\Qx+1.70*cos(\ang-90)},{\Qy+1.70*sin(\ang-90)});
    
\end{scope}

\begin{scope}[shift={(8.6,1.15)}]

    \draw[line width=1pt]
        ({-0.22*\sa},{-0.22*\ca}) -- ({3.3*\sa},{3.3*\ca});
    \draw[line width=1pt]
        ({-0.22*\ca},{0.22*\sa}) -- ({4.90*\ca},{-4.90*\sa});

    \draw[line width=1pt]
        (\Lx,\Ly)
        .. controls (1.15,1.72) and (3.10,1.38) ..
        (\rho,0)
        --
        (\Bx,\By) 
        --
        (0,0)
        --
        cycle;
    \fill[color= gray!20,opacity= .5]
    (\Lx,\Ly)
        .. controls (1.15,1.72) and (3.10,1.38) ..
        (\rho,0)
        --
        (\Bx,\By) 
        --
        (0,0)
        --
        cycle;
    \draw[myred,line width=1pt] (0,-2.7) -- (0,2.80);
    \draw[myred,line width=1pt] (0,0) -- (\rho+0.85,0);
    \draw[myred,dashed,line width=1pt,dash pattern=on 7pt off 5pt]
        (\rho,2.8) -- (\rho,-2.8);
    \node[text=myred!80!black] at (5.5,3.5) {$H_1'$};
    \node[text=myred!80!black] at (0,3.5) {$H_0'$};
    \node[text=myred!80!black] at (6.7,0.1) {$l_0'$};
    \node[text=black] at (2.8,-3) {$L$};


    \draw[myblue,dashed,line width=1.3pt,dash pattern=on 2.5pt off 3.8pt]
        ({\rho-4.5*\ca},{4.5*\sa})
        --
        ({\rho+0.95*\ca},{-0.95*\sa});

\end{scope}
\end{tikzpicture}
\end{center}
\noindent
Denote by $M = L^+ \cup R (L^+)$ and $N = L^- \cup R (L^-) $ where $R$ is the rotation defined above. We repeat the process in \eqref{eq:RS-lp-to-two-bodies-001} and \eqref{eq:RS-lp-to-two-bodies-002} and then \eqref{eq:RS-lp-to-two-bodies-004}, that is, $| M \oplus_p -N|  = \frac{\kappa_{2,p}}{2} (|M|+|N|).$
Since $M$ and $L$ are symmetric with respect to $e_2^\perp$, we write 
\[
    M =\{ (x,s) : - v(x) \leq s \leq v(x), x\in [0,d]\}, \quad N = \{ (y,s) : -w(y) \leq s \leq w(y), y \in [0,d] \},
\]
for some $d >0$.
Using Lemma \ref{lem:decreasing-half-equality} and that the function $w$ is not non-decreasing, it implies that $v$ is identically zero. Thus, $L$ is a simplex with one vertex at the origin, and so is $K$.

Now we will prove the other case. Let $x_0$ be the smallest point in $[0,1)$ at which $f$ attains its maximum value, i.e., $f(x_0)=\|f\|_\infty$. Let $g$ be a function $[0,1]$ defined as follows
\[
        g(z) = \begin{cases}
        f(z), \quad &z \leq x_0,
        \\
        \displaystyle 
        \frac{f(x_0)}{1-x_0} (1-z), \quad &z \geq x_0.
    \end{cases}
\]
Note that $g$ is a concave function, then we denote by $K'$ a convex body corresponding to $g$, that is,
\[
    K' = \{ (x,s) : 0\leq s \leq g(x), x \in [0,1] \}.
\]
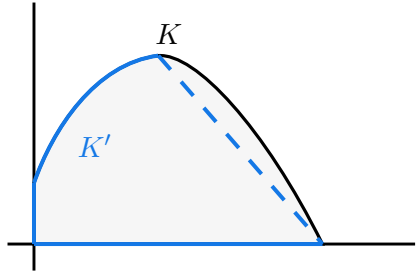
\begin{center}
\begin{tikzpicture}[scale=.7]
    \definecolor{myblue}{RGB}{20,120,230}

    \pgfmathsetmacro{\Py}{1.15}     
    \pgfmathsetmacro{\xzero}{2.35}  
    \pgfmathsetmacro{\yzero}{3.55}  
    \pgfmathsetmacro{\Rx}{5.45}     

    \begin{scope}[shift={(0,0)}]

        \draw[line width=1.1pt] (-0.5,0) -- (7.20,0);
        \draw[line width=1.1pt] (0,-0.5) -- (0,4.55);

        \fill[gray!15,opacity=.45]
            (0,0)
            -- (0,\Py)
            .. controls (0.35,2.15) and (1.15,3.40) .. (\xzero,\yzero)
            .. controls (3.05,3.60) and (4.35,2.20) .. (\Rx,0)
            -- cycle;

        \draw[line width=1.2pt]
            (0,0)
            -- (0,\Py)
            .. controls (0.35,2.15) and (1.15,3.40) .. (\xzero,\yzero)
            .. controls (3.05,3.60) and (4.35,2.20) .. (\Rx,0)
            -- cycle;


        \draw[myblue,line width=1.5pt]
            (0,0) -- (0,\Py);

        \draw[myblue,line width=1.5pt]
            (0,0) -- (\Rx,0);

        \draw[myblue,line width=1.5pt]
            (0,\Py)
            .. controls (0.35,2.15) and (1.15,3.40) .. (\xzero,\yzero);
        \draw[myblue,dashed,line width=1.5pt,dash pattern=on 7pt off 6pt]
            (\xzero,\yzero) -- (\Rx,0);

        \node[text=black] at (2.55,3.95) {$K$};
        \node[text=myblue] at (1.15,1.85) {$K'$};
    \end{scope}
\end{tikzpicture}
\end{center}
Thus, $K' \subset K$ since $f$ is a concave function and so $K' \oplus_p -K' \subset K \oplus_p -K$. Let 
\[
    H^{+}(z,K)  = | \{z +re_2: r \in \R\} \cap (K \oplus_p -K)^+|.
\]
Then $H^{+}(z,K') \leq H^{+}(z,K)$.
Using Lemma \ref{lem:upper-lower-part-lp}, Fubini's theorem and Lemma \ref{lem:formula-h+}, we have
\begin{equation}
\begin{split}
    \frac{\kappa_{2,p}}{2} &= \frac{|K \oplus_p -K|}{2|K|} 
    = \frac{|(K \oplus_p -K)^+| }{|K|} 
    = \frac{
    \int_{-1}^{x_0} H^{+}(z,K) dz + \int_{x_0}^1 H^{+}(z,K) dz
    }{
    \int_{0}^{x_0} f (z) dz + \int_{x_0}^1 f (z) dz
    }
     \\
    &=
    \frac{
    \int_{-1}^{x_0} H^{+}(z,K') dz + \int_{x_0}^1 f(x) dz
    }{
    \int_{0}^{x_0} g (z) dz + \int_{x_0}^1 f (z) dz
    }.
\end{split}
    \label{eq:reduce-last-step-001}
\end{equation}
Thus,
\begin{equation}
\begin{split}
    \int_{-1}^{x_0} H^{+}(z,K') dz &= \frac{\kappa_{2,p}}{2} \int_{0}^{x_0} g (z) dz + \left(\frac{\kappa_{2,p}}{2} -1\right) \int_{x_0}^1 f (z) dz 
    \\
    &\geq \frac{\kappa_{2,p}}{2} \int_{0}^{x_0} g (z) dz + \left(\frac{\kappa_{2,p}}{2} -1\right) \int_{x_0}^1 g (z) dz.
\end{split}
\label{eq:reduce-last-step-002}
\end{equation}
Using  inequality \eqref{eq:rogers_shephard_conjecture},
repeating the equalities \eqref{eq:reduce-last-step-001} for $K'$  and using \eqref{eq:reduce-last-step-002}, we have
\begin{align}
    \frac{\kappa_{2,p}}{2} {\geq}  \frac{|K' \oplus_p K'|}{2|K'|}
    {=} 
    \frac{
    \int_{-1}^{x_0} H^{+}(z,K') dz + \int_{x_0}^1 g(x) dz
    }{
    \int_{0}^{x_0} g (z) dz + \int_{x_0}^1 g (z) dz 
    }
    {\geq} \frac{\kappa_{2,p}}{2}.
\end{align}
Thus, equality holds everywhere. Observe that the equality in \eqref{eq:reduce-last-step-002} implies that $f \equiv g$. Applying a linear transform $T$ such that $Te_1 = e_1$ and $T(x_0,f(x_0)) = (1,f(x_0))$, we obtain that $TK$ is a convex body that lies between $ e_1^\perp$ and $e_1 + e_1^\perp$, $P_{e_2^\perp} (TK) = (TK) \cap e_2^\perp$ and 
\[
    TK = \{ (x,s) : 0 \leq s \leq \zeta(x), x \in [0,1] \},
\]
for some concave and non-decreasing function $\zeta$. The proof is complete using the previous case since $TK$ has to be a triangle with one vertex at the origin, and so is $K$.
\end{proof}

\begin{remark}
    It follows from the equality characterized in Theorem \ref{thm:Planar-case-general} that equality in \eqref{eq:lp-two-bodies} holds if and only if $A^+\cup B^-$ and $A^- \cup B^+$ are triangles with one vertex at the origin. 
\end{remark}

\section*{Appendix.}
Here, we provide the sketch proof of the inequality part of \eqref{eq:rogers_shephard_conjecture} when the convex bodies are restricted to the class of locally anti-blocking convex bodies to present the inequalities that we observed for the equality cases in Theorem \ref{thm:anti-blocking_eq}.
\begin{prop} \label{RS-l-anti}
    Let $K$ be a locally anti-blocking convex body. Then,
    \begin{equation}
        \label{RS-ineq-lp-2}
        | K \oplus_p -K | \leq \sum_{i=0}^n \binom{n/q}{i/q}^{-1}  \binom{n}{i}^2 |K|.
    \end{equation}
\end{prop}

\begin{proof}
    Using \cite[Lemma 2.2]{ASS-23}, for any locally anti-blocking convex bodies $K $, we have
    \begin{equation}
        \label{eq:lemAB-001}
        \left| K \oplus_p - K \right| 
        = \sum_{\delta \in \{ -1,1\} ^n } \left| K_\delta \oplus_p (-K)_\delta \right|.
    \end{equation}
    Using Lemma \ref{strict_containment_implies_strict_bound} and the fact that $(-K)_\delta = -(K_{-\delta})$, we obtain
    \begin{equation}\label{eq:k-delta}
        \left| K_\delta \oplus_p -(K_{-\delta}) \right|
        \leq  
        \left| K_\delta \oplus_p K_{-\delta} \right|.
    \end{equation}
    Since $K_\delta$ and $K_{-\delta}$ are in two opposite orthants, we have from \cite[Lemma 15]{MNZ-25} that
    \begin{equation}
        \label{eq:lemAB-003}
         \left| K_\delta \oplus_p K_{-\delta} \right|
        = \sum_{i=0}^n \sum_{E } \binom{n/q}{i/q}^{-1}
        \left| P_EK_\delta\right| \left| P_{E^\perp} K_{-\delta} \right|,
    \end{equation}
    where the second summation is taken over all $i$-coordinate subspaces.
    
    Fixing $i$-dimensional coordinate subspace $E = \operatorname{span}\{e_i: i\in I\}$, for each $\delta\in \{-1,1\}^n$, it follows from \cite[Fact 2.3]{S-25} and see also in \cite[Proof of the inequality in Theorem 1.1]{S-25}, we claim that there exists unique
    $\tau \in \{ -1,1\}^n$ such that 
    \begin{equation} \label{con-fact-anti}
        P_E K_\delta = P_E K_\tau \text{ and } P_{E^\perp} K_{-\delta} = P_{E^\perp} K_\tau.
    \end{equation}
    Recall the result of Rogers and Shephard \cite{RS-58}: for a convex body $K \subset \mathbb{R}^n$, and for any $k$-dimensional subspace $H$ of $\mathbb{R}^n$,
    \begin{equation}
        \label{lb-RS}
        \left|P_H K\right| \max _{x_0 \in H}\left|K \cap\left(x_0+H^{\perp}\right)\right| \leq\binom{ n}{k} |K|.
    \end{equation}
    Using \eqref{con-fact-anti} and \eqref{lb-RS}, we obtain
    \begin{align}
        \left| K \oplus_p -K \right| &
        \leq
        \sum_{i=0}^n   \sum_{E} \binom{n/q}{i/q}^{-1}
        \sum_{\tau \in \{ -1,1\} ^n } \left| P_EK_\tau\right| \left| P_{E^\perp} K_\tau \right| \label{eq:2}
        \\
        &
        {\leq} \sum_{i=0}^n  \sum_{E } \binom{n/q}{i/q}^{-1}
        \sum_{\tau \in \{ -1,1\} ^n } \binom{n}{i} \left| K_\tau \right|
         \\
        &= \sum_{i=0}^n  \binom{n/q}{i/q}^{-1}
        \binom{n}{i}^2 \left| K \right|,
    \end{align}
    where the summation on $E$ are taken over $i$-dimensional coordinate subspaces.
\end{proof}

\newpage
\bibliographystyle{siam}

\newpage 
\noindent Matthieu Fradelizi
\\
LAMA, Univ Gustave Eiffel, Univ Paris Est Creteil, 77447 Marne-la-Vall\'ee, France.
\\
E-mail address: matthieu.fradelizi@univ-eiffel.fr
\vspace{2mm}
\\
\noindent Auttawich Manui 
\\
Department of Mathematical Sciences, Kent State University, Kent, OH 44242, USA.
\\
E-mail address: amanui@kent.edu
\vspace{2mm}
\\
\noindent Mark Meyer
\\
LAMA, Univ Gustave Eiffel, Univ Paris Est Creteil, 77447 Marne-la-Vall\'ee, France.
\\
E-mail address: mark.meyer@univ-eiffel.fr
\vspace{2mm}
\\
\noindent Cheikh Saliou Ndiaye
\\
LAMA, Univ Gustave Eiffel, Univ Paris Est Creteil, 77447 Marne-la-Vall\'ee, France.
\\
E-mail address: cheikh-saliou.ndiaye@univ-eiffel.fr
\vspace{2mm}
\\

\end{document}